\documentclass{article}
\usepackage{amssymb,amsfonts,amsmath,amsthm}
\usepackage{caption}
\usepackage[export]{adjustbox}
\usepackage{subcaption}
\usepackage{authblk}
\usepackage{color}
\usepackage{esvect}
\usepackage{empheq}
\usepackage{float}
\usepackage{graphicx}
\usepackage{afterpage}
\usepackage{mwe}
\usepackage{comment}
\usepackage{mathtools}
\usepackage{txfonts}
\usepackage{graphicx}

\numberwithin{equation}{section}
\allowdisplaybreaks
\numberwithin{equation}{section} 

\newtheorem{thm}{Theorem}[section]
\newtheorem{defn}[thm]{Definition}
\newtheorem{cor}[thm]{Corollary}
\newtheorem{lemma}[thm]{Lemma}

\newtheorem{rmrk}[thm]{Remark}

\newcommand{\R}{\mathbb{R}}

\newtheorem{prop}[thm]{Proposition}

\newcommand{\abs}[1]{\left\vert{#1}\right\vert}

\newcommand{\ba}{\begin{array}}
\newcommand{\ea}{\end{array}}

\newcommand{\bthm}{\begin{thm}}
\newcommand{\ethm}{\end{thm}}
\newcommand{\bstp}{\begin{stp}}
\newcommand{\estp}{\end{stp}}
\newcommand{\blemma}{\begin{lemma}}
\newcommand{\elemma}{\end{lemma}}
\newcommand{\bprop}{\begin{prop}}
\newcommand{\eprop}{\end{prop}}
\newcommand{\bpf}{\begin{pf}}

\newcommand{\epf}{\end{pf}}
\newcommand{\bdefn}{\begin{defn}}

\newcommand{\edefn}{\end{defn}}
\newcommand{\brk}{\begin{rmrk}}
\newcommand{\erk}{\end{rmrk}}
\newcommand{\bcrl}{\begin{crl}}
\newcommand{\ecrl}{\end{crl}}
\newcommand{\norm}[1]{\left\|#1\right\|}

\newcommand{\beqn}{\begin{equation}}
\newcommand{\eeqn}{\end{equation}}

\renewcommand{\leq}{\leqslant}
\renewcommand{\geq}{\geqslant}

\newcommand{\beq}{\begin{equation}}
\newcommand{\eeq}{\end{equation}}
\newcommand{\bea}{\begin{eqnarray}}

\newcommand{\eea}{\end{eqnarray}}

\newcommand{\N}{\mathcal{N}}
\newcommand{\dl}{\delta}
\newcommand{\gm}{\gamma}
\newcommand{\delp}{p^{\delta}}
\newcommand{\delv}{v^{\delta}}
\newcommand{\delpo}{p_0^{\delta}}
\newcommand{\delpsi}{\psi^{\delta}}
\newcommand{\delw}{w^{\delta}}

\begin{document}
\renewcommand\Authfont{\small}
\renewcommand\Affilfont{\itshape\footnotesize}

\author[1]{Jacob Rubinstein\footnote{koby@technion.ac.il}}
\author[2]{Peter Sternberg\footnote{sternber@indiana.edu}}
\affil[1]{Department of Mathematics, Technion, I.I.T., 32000 Haifa, Israel}
\affil[2]{Department of Mathematics, Indiana University, Bloomington, IN 47405}
\title{Reduced equations for an active model of hydroelastic waves in the cochlea}

\maketitle

\noindent {\bf Abstract:} Building upon our earlier passive models for the cochlea \cite{rust1,rust2}, here we enhance the model with an active mechanism. Starting with a one-chamber simplification leading to a system of a time-dependent PDE in two spatial variables for the pressure coupled to a PDE in one spatial variable for the oscillation of the basilar membrane, we rigorously establish the validity of a dimension reduction to a system to two ODE's. We then present numerical simulations demonstrating the ability of this reduced active system to distinguish and amplify multi-frequency input signals.

\date



\section{Introduction}  \setcounter{equation}{0}

One of the fascinating properties of the cochlea is the existence of an active mechanism that amplifies the vibrations of the basilar membrane (BM). There are numerous effects of this amplification, including autoemission of sound, the ability to capture very weak signals, fine tuning of frequency selectivity, and compressibility, namely a nonlinear response curve including saturation \cite{ashmore10}, \cite{hudspeth14}. Therefore, it is required to supplement the ``standard" models of sound detection, including the fluid-solid wave in the cochlea and the vibration of the BM, with a nonlinear active term.

While a lot of information was collected over the years about this mechanism, including in particular a candidate protein to activate it \cite{dallos92}, \cite{zheng00}, it is still not fully deciphered. A number of theories were proposed e.g. \cite{yoon09}, \cite{yoon11},  \cite{geisler91} , \cite{geisler95}, and more. An attractive concept was introduced by Hudspeth and his colleagues. They observed \cite{hudspeth14} that the classical Hopf bifurcation implies many of the instabilities that are associated with the cochlea active mechanism. In fact, they even identified experimentally \cite{howard88} a regime of negative stiffness in the basilar membrane response that might be modeled as a Hopf bifurcation. While the Hopf picture is atractive, however, it has a few drawbacks. For example, the nonlinearity depends on the amplitude of the BM. This might affect the {\em place principle}, although it is known that this principle, where each point along the cochlea is tuned to a specific frequency, is only slightly affected by the nonlinear active mechanism. Also, while the negative stiffness reported in \cite{howard88} is a very interesting observation, it occurs over small BM amplitudes, quite smaller than the full nonlinear regime.

As an alternative, we therefore propose here a different self-oscillatory model that depends on the {\em velocity} of the BM vibration. Denoting the BM vibration by $v(x,y,t)$ where $x$ is a coordinate along the cochlea and $y$ is a coordinate orthogonal to it in the BM rest plane, a nonlinear term of the form $(1-\dot{v}^2)\dot{v}$, where $\cdot=\frac{\partial}{\partial t}$, gives rise to the well-known Rayleigh oscillator, which can be shown to yield all the basic effects of the cochlea active mechanism. However, the Rayleigh nonlinear term is unbounded in $\dot{v}$. Therefore we propose here a modified Rayleigh term such as $\rho\dot{v}e^{-c\abs{\dot{v}}}$ or $\tanh(\rho \dot{v})$ to model the BM nonlinear response.

The cochlea is a highly complex organ. Therefore most mathematical models of it are based on simplifications where the central part of it, namely the scala media, is neglected, and the attention is focused at the BM as a single partition between the scala vestibuli and the scala timpani. It is further assumed that the cochlea response can be modeled solely through the interaction of the fluid on the two sides of the BM and the BM itself. Many models go further to use the large aspect ratio of the BM to neglect the cross section variable $y$ and the channel height variable $z$ and thus write down a one-dimensional reduced model for the BM.

 In \cite{rust1} we examined the last approximation for a passive cochlea model. Using a Fourier representation of the underlying fluid and solid equations we proved the validity of the one-dimensional model under certain assumptions. It is interesting to note, as was indeed observed in \cite{rust1}, that while neglecting the $y$ coordinate is quite straightforward, neglecting the $z$ variable is not trivial since the wavelength of the wave traveling along the cochlea may become comparable to the channel height precisely at the critical location for each frequency. While the model in \cite{rust1} is based on a simple spring elastic model, we proved later \cite{rust2} a similar result for a more elaborate elastic model where the BM is taken to be a membrane. Our goal in this paper is to justify the one-dimensional approximation also for the nonlinear active model mentioned above. A major difference between the present work and our two earlier works is the nonlinearity in the model, that among other things prevents the use of the Fourier representation.

In the next section we introduce our equations. To simplify the presentation, and since as pointed out above even the more elaborate three-dimensional models are a caricature of the full cochlea dynamics, we employ a simple model where only the fluid partition above the BM is considered. Furthermore, again for simplicity, we neglect the cross-sectional variable $y$. Both these simplifications can be removed without changing our main result. In section 3 we derive a priori estimates that are later used in section 4 to pass to the limit $\dl \rightarrow 0$ where $\dl$ is the small nondimensional thickness and height. In section 5 we present a few simulations  to demonstrate the effect of the nonlinearity. Finally, we supply an appendix where we identify the fluid-solid energy term and provide formally an alternative proof of the reduced model limit.

\section{A Model for the BM with an Active Mechanism}  \setcounter{equation}{0}

We will pursue a model in which the motion of the basilar membrane is captured as a spring with mass $m$, damping constant $r$ and a variable spring constant $k(x)$, where $x$ denotes a variable measuring distance along the cochlea. This model is similar to the one analyzed in \cite{rust1} in that we assume the cochlea is filled with a linear, ideal fluid, leading to the assumption that the pressure is harmonic, cf. \eqref{Laplace} below. However,
as mentioned in the introduction, the model here differs from that given in \cite{rust1} in that we postulate here an active term so that the oscillations of the vertical deflection $v$ of the membrane are enhanced by a nonlinear forcing term $\mathcal{N}=\mathcal{N}(\dot{v})$ depending on the velocity of the deflection, cf. \eqref{spring}.
Another difference from \cite{rust1} is that here, for simplicity, we will take a one chamber model and we will ignore variations in the plane of the BM orthogonal to the $x$-direction,  in order to focus on the effect of the active mechanism without extra complications. We assume an aspect ratio of $\delta\ll1 $ between the vertical and longitudinal dimensions of the cochlea and so we scale by $\delta$ in the transverse $z$-direction (hence the appearance of $\delta^{-2}$ in \eqref{Laplace}) and assume the deflection is $O(\delta)$ . Thus, the vertical deflection $v$  appearing below is really the original deflection divided by $\delta$ (hence the appearance of $\delta^2$ in \eqref{topbotbc}) . With these scalings, our model
involves a pressure $p$ and a deflection $v$ that depend on spatial variables $x\in [0,1],\,z\in [0,1]$ and time $t$. Finally, we take the system to be driven by a specified input at the oval window (here $x=0$), that we denote by $f=f(t)$. Our system is then given by:
\begin{eqnarray}
&&
p_{xx}+\frac{1}{\delta^2}p_{zz}=0\;\mbox{for}\;0<x<1,\;0<z<1,\;t>0,\label{Laplace}\\
&&p(0,z,t)=f(t),\;p(1,z,t)=0\;\mbox{for}\;0<z<1,\;t>0,\label{sidebc}\\
&&p_z(x,1,t)=0,\;p_z(x,0,t)=-\delta^2 \ddot{v}(x,t)\;\mbox{for}\;0<x<1,\;t>0,\label{topbotbc}\\
&& m\ddot{v}+r\dot{v}+k(x)v=-p(x,0,t)+\N(\dot{v})\;\mbox{for}\;0<x<1,\;z=0,\;t>0,\label{spring}\\
&&v(x,0)=0,\;\dot{v}(x,0)=0\;\mbox{for}\;0<x<1\label{ic},
\end{eqnarray}
where $m,r$ and $k(x)$ are positive. The second condition in \eqref{topbotbc} arises from equating the acceleration of the ideal fluid to the acceleration of the basilar membrane where it makes contact along the bottom boundary. Again, we refer the reader to \cite{rust1} for more details of the derivation.

To model the active mechanism, we choose the nonlinearity $\N$ to satisfy the conditions
\begin{equation}
\mathcal{N}(0)=0,\quad \norm{\mathcal{N}}_{L^{\infty}(\R)}<\infty,\quad
 \norm{\mathcal{N'}}_{L^{\infty}(\R)}\leq \rho\;\mbox{for some constant}\;\rho<r.\label{Nprops}
\end{equation}
Possible choices include, for example,
\begin{equation}
\mathcal{N}(\dot{v})=\rho \dot{v}e^{-c\abs{\dot{v}}}\quad\mbox{or}\quad
\mathcal{N}(\dot{v})
=\tanh(\rho \dot{v}).\label{Nrho}
\end{equation}

We should stress that for now we are suppressing the dependence of $p$ and $v$ on $\delta$ for ease of notation.



\vskip 0.2cm
\noindent {\bf Existence/Regularity of the Solution to \eqref{Laplace}--\eqref{ic}}.
\vskip 0.1cm
\noindent The problem \eqref{Laplace}--\eqref{ic} is a hybrid of a PDE in the spatial variables $x$ and $z$ with an ODE in time. Through a fixed
point argument that we will not present, one can establish the existence of a local in time solution that will be smooth in the interior of the square $(x,z)\in (0,1)\times (0,1)$ and continuous up to the boundary. Then through a priori estimates similar to the ones we present below, this solution can be extended to exist for all $t>0$. We note, however, that due to the incompatibility at the corners $(x,z)=(0,0)$ and $(x,z)=(1,0)$ between the $z$-derivative of \eqref{sidebc}, namely $p_z(0,0,t)=0=p_z(1,0,t)$, and the bottom boundary condition $p_z(0,0,t)=-\delta^2\ddot{v}(0,t)$ and $p_z(1,0,t)=-\delta^2\ddot{v}(1,t)$  with $\ddot{v}$ given by \eqref{spring}, the solution will necessarily have unbounded derivatives as one approaches these corners.

\section{$\delta$-Independent Estimates on the Solution to \eqref{Laplace}--\eqref{ic}}

Our first aim is to establish bounds independent of $\delta$ on the solution pair $(p,v)$. To this end, we will first invoke the variation of constants formula for ODE's to rephrase \eqref{spring}-\eqref{ic} as an integral equation. Such a formulation of course involves a fundamental set of solutions, say $\{v_1,v_2\}$ to the homogeneous problem
\[
m\ddot{v}+r\dot{v}+k(x)v=0.
\]
One has
\begin{align*}
&v_1(x,t)=e^{-\frac{r}{2m}t}\cos\left(\frac{\sqrt{4mk(x)-r^2}}{2m}\,t\right),\;
v_2(x,t)=e^{-\frac{r}{2m}t}\sin\left(\frac{\sqrt{4mk(x)-r^2}}{2m}\,t\right)\quad
\mbox{when}\;4mk(x)>r^2,\\
&v_1(t)=e^{-\frac{r}{2m}t},\;v_2(t)=te^{-\frac{r}{2m}t}\quad\mbox{when}\;4mk(x)=r^2,\\
&v_1(x,t)=e^{\lambda_+(x)\,t},\;v_2(x,t)=e^{\lambda_-(x)\,t}\;\mbox{where}\;\lambda_{\pm}(x)=\frac{-r\pm\sqrt{r^2-4k(x)m}}{2m}
\quad \mbox{when}\;4mk(x)<r^2.
\end{align*}
Then one can write down an integral equation satisfied by $v$ that is equivalent to \eqref{spring}-\eqref{ic} using the standard variation of constants formula:
\begin{align*}
v(x,t)=&\frac{1}{m}\left\{-v_1(t)\int_0^t \frac{v_2(s)\big(-p(x,0,s)+\N(\dot{v}(x,s))\big)}{W(v_1(s),v_2(s))}\,ds\right\}\\
&+\frac{1}{m}\left\{v_2(t)\int_0^t \frac{v_1(s)\big(-p(x,0,s)+\N(\dot{v}(x,s))\big)}{W(v_1(s),v_2(s))}\,ds
\right\},
\end{align*}
where $W(v_1,v_2)$ denotes the Wronskian $v_1(v_2)_t-v_2(v_1)_t.$

Differentiating twice with respect to $t$ one finds that in all three cases above,  $\ddot{v}$ satisfies an equation of the form
\begin{equation}
\ddot{v}(x,t)=\int_0^tK(x,t-s)\left\{-p(x,0,s)+\N(\dot{v}(x,s))\right\}\,ds
+\frac{1}{m}\left(-p(x,0,t)+\N(\dot{v}(x,t))\right),\label{vtt}
\end{equation}
for a kernel $K(x,t).$
Furthermore, regardless of the sign of $4k(x)m-r^2$ one finds that
\begin{equation}
\abs{K(x,t)}\leq \gamma\;\mbox{for all}\;x\in[0,1]\;\mbox{and}\;t\geq 0,\label{absK}
\end{equation}
in light of the exponential decay of $v_1$ and $v_2$ in all cases.

 Now introducing $q(x,z,t)$ via $p(x,z,t)=q(x,z,t)+f(t)(1-x)$, we see from \eqref{Laplace}, \eqref{sidebc} and \eqref{topbotbc} that $q$ satisfies the equation
 \begin{equation}
 q_{xx}+\frac{1}{\delta^2}q_{zz}=0\;\mbox{for}\;0<x<1,\;0<z<1,\;t>0,\label{QLaplace}\\
 \end{equation}
 and the boundary conditions
\begin{eqnarray}
&&q(0,z,t)=0=q(1,z,t)=0\;\mbox{for}\;0<z<1,\;t>0\label{Qsidebc}\\
&&q_z(x,1,t)=0,\;q_z(x,0,t)=-\delta^2 \ddot{v}(x,t)\;\mbox{for}\;0<x<1,\;t>0.\label{Qtopbotbc}
\end{eqnarray}

If we multiply \eqref{QLaplace} by $-q$ and integrate by parts, a use of the boundary conditions \eqref{Qsidebc} and \eqref{Qtopbotbc}
leads to the identity
\begin{equation}
\int_0^1\int_0^1 \left({q_{x}}^2+\frac{1}{\delta^2} {q_{z}}^2\right)\,dx\,dz=\int_0^1 q(x,0,t)\,\ddot{v}(x,t)\,dx,\label{ibp}
\end{equation}
and so substituting in \eqref{vtt} we find that
\begin{eqnarray}
&&\int_0^1\int_0^1 \left({q_{x}}^2+\frac{1}{\delta^2} {q_{z}}^2\right)\,dx\,dz=\nonumber\\
&&
\int_0^1\int_0^t q(x,0,t)K(x,t-s)\left\{-q(x,0,s)-f(s)(1-x)+\N(\dot{v}(x,s))\right\}\,ds\,dx\nonumber\\
&&+\frac{1}{m}\int_0^1q(x,0,t)\left(-q(x,0,t)-f(t)(1-x)+\N(v(x,t))\right)\,dx.
\end{eqnarray}
After collecting the nonnegative terms on the left, we conclude that
\begin{eqnarray}
&&\int_0^1\int_0^1 \left({q_{x}}^2+\frac{1}{\delta^2} {q_{z}}^2\right)\,dx\,dz+\frac{1}{m}\int_0^1q(x,0,t)^2\,dx\nonumber\\
&&=-\int_0^1\int_0^t q(x,0,t)K(x,t-s)q(x,0,s)\,dx\,dx\nonumber\\
&&+\int_0^1\int_0^t q(x,0,t)K(x,t-s)\left\{-f(s)(1-x)+\N(\dot{v}(x,s))\right\}\,ds\,dx\nonumber\\
&&+\frac{1}{m}\int_0^1q(x,0,t)\left\{-f(t)(1-x)+\N(\dot{v}(x,t))\right\}\,dx.\label{id1}
\end{eqnarray}

{\bf Closing the estimates for $T<C(m,k,r)$:}\\

With an eye towards bounding the $t$-integral of the left-hand side of \eqref{id1} we then use \eqref{absK} to estimate that for any $T>0$:
\begin{eqnarray}
&&\abs{\int_0^1\int_0^T\int_0^t q(x,0,t)K(x,t-s)q(x,0,s)\,ds\,dt\,dx}\leq \gamma \int_0^1\int_0^T\int_0^t \abs{q(x,0,t)}\abs{q(x,0,s)}\,ds\,dt\,dx\nonumber\\
&&=\frac{\gamma}{2}\int_0^1\int_0^T\frac{\partial}{\partial t}\left(\int_0^t \abs{q(x,0,s}\,ds\right)^2\,dt\,dx=
\frac{\gamma}{2}\int_0^1\left(\int_0^T \abs{q(x,0,t)}\,dt\right)^2\,dx\nonumber\\
&&\leq \frac{\gamma}{2} T\int_0^1 \int_0^T q(x,0,t)^2\,dt\,dx.\nonumber\\\label{Test}
\end{eqnarray}
Given that in addition to $K$, the functions $f$ and $\N$ are uniformly  bounded on $[0,T]$, the $t$ integrals of the last two integrals on the right-hand side of \eqref{id1} are easily controlled via
Cauchy-Schwarz as
\begin{eqnarray}
&&\abs{\int_0^T\int_0^1\int_0^t q(x,0,t)K(x,t-s)\left\{-f(s)(1-x)+\N(\dot{v}(x,s))\right\}\,ds\,dx\,dt}\nonumber\\
&&+\abs{\frac{1}{m}\int_0^T\int_0^1q(x,0,t)\left\{-f(t)(1-x)+\N(\dot{v}(x,t))\right\}\,dx\,dt}\nonumber\\
&&\leq Const(\norm{f}_{L^\infty([0,T]},\norm{\mathcal{N}}_{L^\infty(\R)},m,r,k,T)\left(\int_0^T\int_0^1 q(x,0,t)^2\,dx\,dt\right)^{1/2}.\label{lasttwo}
\end{eqnarray}

Then integrating \eqref{id1} with respect to $t$ and using \eqref{Test} and \eqref{lasttwo} we find
\begin{eqnarray}
&&\int_0^T \int_0^1\int_0^1 \left({q_{x}}^2+\frac{1}{\delta^2} {q_{z}}^2\right)\,dx\,dz\,dt+\left(\frac{1}{m}-\frac{\gamma}{2}T\right)
\int_0^T\int_0^1q(x,0,t)^2\,dx\,dt \nonumber\\
&&\leq \int_0^T \int_0^1\int_0^t q(x,0,t)K(x,t-s)\N(\dot{v}(x,s))\,ds\,dx\,dt
+\frac{1}{m}\int_0^T\int_0^1q(x,0,t)\N(\dot{v}(x,t))\,dx\,dt\nonumber\\
&&\leq Const(\norm{f}_{L^\infty([0,T]},m,r,k,T)\left(\int_0^T\int_0^1 q(x,0,t)^2\,dx\,dt\right)^{1/2}.\label{Pbd}
\end{eqnarray}
It then follows that up until time $T$ where $\left(\frac{1}{m}-\frac{\gamma}{2}T\right)>0$, that is, provided that say
\begin{equation}
T\leq T^*:=\frac{1}{\gamma m},\label{smallT}
\end{equation}
one has a $\delta$-independent bound on the quantities
\[
\int_0^T \int_0^1\int_0^1 \left({q_{x}}(x,z,t)^2+\frac{1}{\delta^2} {q_{z}}(x,z,t)^2\right)\,dx\,dz\,dt\quad\mbox{and}\quad
\int_0^T\int_0^1q(x,0,t)^2\,dx\,dt,
\]
with the upper bound depending on $m,r$ and $k$. Then since
$q(0,z,t)=0$ and $q(1,z,t)=0$, we can invoke the Poincar\'e inequality
\[
\int_0^1 q(x,z,t)^2\,dx\leq \frac{1}{\pi^2}\int_0^1 {q_x(x,z,t)}^2\,dx\quad\mbox{for every}\;z\in (0,1)\;\mbox{and}\;t>0.
\]
Integrating this inequality with respect to $z$ and $t$ as well and appealing to \eqref{Pbd}, we also obtain a $\delta$-independent bound on
\[
\int_0^T \int_0^1\int_0^1 q(x,z,t)^2\,dx\,dz\,dt
\]
depending on $m,r,k$ and $\norm{f}_{L^\infty([0,T]}$, provided $T$ satisfies \eqref{smallT}. Since $p=q+f(1-x)$, we obtain similar bounds on $p$ and in particular, we conclude that for a.e. $t\in [0,T^*]$,
the function $p=p^{\delta}(\cdot,\cdot,t)$ is bounded in $H^1\big((0,1)\times(0,1)\big)$ independent of $\delta$. Then, since $T^*$ does not depend on initial data, we can simply repeat this estimate, starting with $v(\cdot,T^*),\,\dot{v}(\cdot,T^*)$ as initial data in \eqref{spring}. Since $\ddot{v}$ as characterized in \eqref{vtt} is independent of $v(\cdot,T^*)$ and $\dot{v}(\cdot,T^*)$ we obtain:
\begin{prop} \label{boundy}For every $T>0$
there exists a constant $C_0=C_0(\norm{f}_{L^\infty([0,T]},m,r,k,T)$\\ independent of $\delta$ such that
\begin{equation}
\int_0^T \int_0^1\int_0^1 \left(p^2+{p_{x}}^2+\frac{1}{\delta^2} {p_{z}}^2\right)\,dx\,dz\,dt\;+\;
\int_0^T\int_0^1p(x,0,t)^2\,dx\,dt<C_0.\label{psbd}
\end{equation}
\end{prop}
\begin{cor}\label{vbds} For every $T>0$ there exists a constant $C_1=C_1(\norm{f}_{L^\infty([0,T]},m,r,k,T)$ independent of $\delta$ such that
\beq
 \int_0^1 \left(\dot{v}^2(x,t) +  v^2(x,t) \right) \; dx+
 \int_0^1 \int_0^T  \dot{v}^2(x,t) \, dt\,dx
 \leq C_1\;\mbox{ for all}\;t\in [0,T].\label{spaceint}
 \eeq
\end{cor}
\noindent
{\bf Proof of Corollary}: We multiply equation \eqref{spring} by $\dot{v}$ and integrate in $x$ and $t$ up to any time $T>0$:
\[
\int_0^1 \int_0^T \left( m \ddot{v} \dot{v} + r \dot{v}^2 + k(x)v \dot{v} \right)\; dt\,dx
=\int_0^1 \int_0^T\bigg\{ \N(\dot{v})-p(x,0,t)  \bigg\}\,\dot{v}(x,t)\;\,dt\,dx.
\]
 Using the initial condition \eqref{ic} we can then integrate two of the terms on the left with respect to time to obtain
\begin{eqnarray}
&&\frac{1}{2} \int_0^1 \left(m \dot{v}^2(x,T) + k(x) v^2(x,T) \right) \; dx + \int_0^1 \int_0^T r \dot{v}^2(x,t) \,dt\,dx=  \nonumber \\
&&\int_0^1 \int_0^T\bigg\{ \N(\dot{v})-p(x,0,t)  \bigg\}\,\dot{v}(x,t)\;\,dt\,dx .\label{e11}
\end{eqnarray}
Bounding the integrand on the right via $\{\,\cdot\,\}\dot{v}\leq \frac{1}{2r}\{\,\cdot\,\}^2+\frac{r}{2}\dot{v}^2$ so that we can absorb $\frac{r}{2}\int\int \dot{v}^2$
onto the left-hand side, we then invoke \eqref{Nrho} and \eqref{psbd} to reach \eqref{spaceint}. \qed

\section{Passing to the Limit $\delta\to 0$}

In this section, we denote explicitly the $\delta$-dependence of the solutions $\delp$ and $\delv$ to \eqref{Laplace}--\eqref{ic}.
Let us now define the integral average of $\delp$ via
\begin{equation}
\delpo(x,t):=\int_0^1\delp(x,z,t)\,dz.\label{pave}
\end{equation}

Then integrating \eqref{Laplace} with respect to $z$ and using the boundary conditions \eqref{topbotbc} we find that $\delpo$
satisfies the equation
\begin{equation}
(\delpo)_{xx}(x,t)=\delp_z(x,0,t)= -\ddot{v}^\delta(x,t),
\label{intdelp}
\end{equation}
and an integration of the boundary conditions \eqref{sidebc} yields
\begin{equation}
\delpo(0,t)=f(t),\quad \delpo(1,t)=0.
\label{delpobc}
\end{equation}
Our aim is to couple this problem for $\delpo$ with the ODE \eqref{spring} to get a closed system in $\delpo$ and $\delv$ alone and for this we will invoke the following simple lemma.
\blemma\label{closepp}
For every $T>0$ one has the bounds
\begin{equation}
\int_0^T\int_0^1\abs{\delpo(x,t)-\delp(x,0,t)}^2\,dx\,dt<C_0\,\delta^2,\label{ebd}
\end{equation}
 and
\begin{equation}
\int_0^T\int_0^1\int_0^1\abs{p^\delta(x,z,t)-\delpo(x,t)}^2\,dx\,dz\,dt<\frac{C_0}{\pi^2}\,\delta^2,\label{ebd2}
\end{equation}
where $C_0$ is the constant from Proposition \ref{boundy} that is independent of $\delta$.
\elemma
\begin{proof}
We have
\[
\abs{\delpo(x,t)-\delp(x,0,t)}=\abs{\int_0^1\int_0^z(\delp)_z(x,\eta,t)\,d\eta\,dz}\leq \int_0^1\abs{(\delp)_z(x,z,t)}\,dz
\]
from which it follows that
\begin{eqnarray*}
\int_0^1 \abs{\delpo(x,t)-\delp(x,0,t)}^2\,dx&&\leq \int_0^1\left(\int_0^1\abs{(\delp)_z(x,z,t)}\,dz\right)^2\,dx\\
&&\leq \int_0^1\int_0^1\abs{(\delp)_z(x,z,t)}^2\,dz\,dx.
\end{eqnarray*}
The estimate \eqref{ebd} then follows after an integration in $t$ once we invoke Proposition \ref{boundy}. The bound
\eqref{ebd2} follows by a similar argument once one applies the Poincar\'e inequality to the quantity
$p^\delta(x,z,t)-\delpo(x,t)$ which has mean zero when integrated with respect to $z$.

\end{proof}

Introducing the ``error term"
\begin{equation}
E^\delta(x,t):=\delpo(x,t)-\delp(x,0,t)
\quad\mbox{which satisfies}\quad\norm{E^\delta}_{L^2((0,1)\times(0,T))}=O(\delta)\label{smallE}
\end{equation}
 from \eqref{ebd}, we then reach the system \eqref{intdelp}-\eqref{delpobc} coupled to the ODE
\begin{equation}
 m\ddot{v}^\delta+r\dot{v}^\delta+k(x)\delv=-\delpo+\N(\dot{v}^\delta)+E^\delta\quad\mbox{for}\;0<x<1,\;t>0,\label{spring2}
\end{equation}
supplemented by the initial conditions \eqref{ic}.

We conclude this section with our main result on the $\delta\to 0$ limit of the pair $(\delpo,\delv).$
\begin{thm} \label{main} Denote by $p_1=p_1(x,t)$ and $v_1=v_1(x,t)$ the pair of functions solving the
system
\begin{eqnarray}
&&{p_1}_{xx} = -\ddot{v}_1\quad \mbox{for}\;0<x<1,\;t>0\label{pbar}\\
&&\mbox{subject to boundary conditions}\quad p_1(0,t)=f(t),\;p_1(1,t)=0\label{pbarbc}\\
&&\mbox{and}\nonumber\\
&&m\ddot{v}_1 + r\dot{v}_1 + k(x) v_1 = -p_1 +\N(\dot{v})\quad\mbox{for}\;0\leq x\leq 1,\;t>0\label{vbar}\\
&&\mbox{subject to initial conditions}\quad
 v_1(x,0)=\dot{v}_1(x,0)=0.\label{vbaric}
\end{eqnarray}
Then for any time $T>0$ there exists a constant $C$ independent of $\delta$ such that
\begin{equation}
\int_0^T\int_0^1\bigg\{\big(\delpo-p_1\big)^2+\big(\delv-v_1\big)^2+
\big((\delpo)_x-{p_1}_x\big)^2+\big(\dot{v}^\delta-\dot{v}_1\big)^2\bigg\}\,dx\,dt\leq C\delta^2\label{converges}
\end{equation}
and
\begin{equation}
\int_0^T\int_0^1\int_0^1 \big(p^\delta(x,z,t)-p_1(x,t)\big)^2\,dx\,dz\,dt<C\delta^2 .\label{alsoconverges}
\end{equation}
\end{thm}
\begin{proof}
We begin by introducing the quantities $\delpsi:=\delpo-p_1$ and $\delw:=\delv-v_1$. Writing
\[
\N(\delv)=\N(v_1+\delw)=\N(v_1)+\N'(h^\delta)\delw
\]
for some function $h^\delta=h^\delta(x,t)$ taking values between $v_1$ and $\delv$, it follows from \eqref{intdelp} and \eqref{spring2} that
the pair $(\delpsi,\delw)$ solves the system
\begin{equation}
 m\ddot{w}^\delta+r\dot{w}^\delta + k(x)\delw = -\delpsi + N'(h^\delta)\dot{w}^\delta +E^{\delta}\label{one}
\end{equation}
\begin{equation}
\; \delpsi_{xx} = -\dot{w}^\delta =\frac{1}{m}\bigg( \delpsi -  N'(h^\delta) \dot{w}^\delta  +r\dot{w}^\delta + k(x)\delw -E^{\delta}\bigg).\label{two}
\end{equation}
along with boundary and initial conditions
\begin{equation}
\delpsi(0,t)=\delpsi(1,t)=0\quad\mbox{and}\quad \delw(x,0)=\dot{w}^\delta(x,0)=0.\label{bcicpw}
\end{equation}
 We then multiply \eqref{one} by $\dot{w}^\delta$ and integrate over $[0,1]\times [0,T]$ to see that
\begin{eqnarray}
&&\int_0^1 \frac{1}{2}\left((\dot{w}^\delta)^2(x,T) + k(x) (\delw)^2(x,T)\right)\,dx+ \int_0^T\int_0^1 r (\dot{w}^\delta)^2 \,dx\,dt+
\int_0^T\int_0^1 \delpsi\, \dot{w}^\delta \,dx\,dt\nonumber\\
&&=\int_0^T\int_0^1\left(\N'(h^\delta) (\dot{w}^\delta)^2 + E^\delta \dot{w}^\delta\right)\,dx\,dt.\label{iden}
\end{eqnarray}
Regarding the last term on the left-hand side of
\eqref{iden} we appeal to \eqref{two} and \eqref{bcicpw} to see that it is nonnegative since
\begin{eqnarray}
&&
\int_0^T\int_0^1 \delpsi(x,t)\dot{w}^\delta(x,t)\,dx\,dt= -\int_0^T\int_0^1 \delpsi(x,t)\int_0^t \delpsi_{xx}(x,s)\,ds\,dx\,dt\nonumber \\
&&= -\int_0^T\int_0^t\int_0^1 \delpsi(x,t)\delpsi_{xx}(x,s)\,dx\,ds\,dt= \int_0^T\int_0^t\int_0^1 \delpsi_x(x,t)\delpsi_{x}(x,s)\,dx\,ds\,dt \nonumber \\
&&=\int_0^1\int_0^T H(x,t)\,\frac{\partial}{\partial t} H(x,t)\,dt\,dx=\frac{1}{2} \int_0^1 H(x,T)^2\,dx=\frac{1}{2} \int_0^1 \left(\int_0^T\psi_x(x,t)\,dt\right)^2\,dx,\nonumber\\
&&\label{short}
\end{eqnarray}
where $H(x,t):=\int_0^t \delpsi_x(x,s)\,ds.$

Returning to \eqref{iden} we can then use that $\abs{\N'(h^\delta)}\leq \rho<r$ from assumption \eqref{Nprops} along with the estimate \eqref{smallE} to derive the inequality for every $t\in [0,T]$:
\begin{equation}
 \int_0^1 (\dot{w}^\delta)^2(x,t) + (\delw)^2(x,t)\,dx+ \int_0^T\int_0^1 (\dot{w}^\delta)^2(x,t) \,dx\,dt
\leq C\delta^2\label{three}
\end{equation}
for a constant $C$ independent of $\delta$. Integrating this bound with respect to $t$ over $[0,T]$ we obtain the $L^2\big([0,1]\times [0,T]\big)$ bounds on
$\delv-v_1$ and $\dot{v}^\delta-\dot{v}_1$ claimed in \eqref{converges}. Finally, we multiply \eqref{two} by $-\delpsi$ and integrate
over $[0,1]\times [0,T]$. After an integration by parts in $x$ we obtain the $L^2$ bounds on $\delpo-p_1$ and $(\delpo)_x-{p_1}_x$ claimed in \eqref{converges} through the use of  \eqref{smallE} and \eqref{three}. Inequality \eqref{alsoconverges} then follows from \eqref{ebd2} and \eqref{converges}.

\end{proof}

\section{Discussion}

We have carried out a rigorous analysis to demonstrate that the model \eqref{Laplace}--\eqref{ic} involving two spatial dimensions and time can be well-approximated by the much simpler reduced model \eqref{pbar}-\eqref{vbaric} in one spatial variable and time. From the standpoint of this analysis, the fine detail of the nonlinearity $\mathcal{N}$ is not crucial, so long as the conditions in \eqref{Nprops} are met. However, in testing the effectiveness of the model numerically, we have opted  below for the first choice $\rho \dot{v}e^{-c\abs{\dot{v}}}$ in \eqref{Nrho} as opposed to the second choice $\tanh(\rho\dot{v})$ since the latter saturates at large values of $\dot{v}$, thus pumping energy into the system. That said, the computational results are not dramatically different for these two forms of $\N$.

To demonstrate the effects of the nonlinear active model we simulated it once in the case of no amplification $\N=0$, (Figure 1, left), and once for the one-dimensional reduced active model (Figure 1, right). In our simulations we normalized the cochlea length to be $1$. The time unit is chosen so that the input signal takes the form $0.1\cos(2t)+0.08\cos(2.4t)$, corresponding to a two-tone signal at frequencies $2$KHz and $2.4$KHz. The mass $m$ and the stiffness $k(x)$ were selected to fit the place principle $k(x)=400m\exp(-9.6 x)$. We use the nonlinear term
$\mathcal{N}(\dot{v})=\rho\dot{v} \exp(-0.05|\dot{v}|)$, cf. \eqref{Nrho}. The friction parameter was set somewhat arbitrarily to $r=0.3$. Finally we set $\rho=0.2995$. In both models we observe the best response at the locations dictated by the place  principle, as anticipated. However, the signal in the nonlinear case is amplified by a factor of about $10$  compared to the passive case. (Note the different vertical scales on the left and right.) Furthermore, the signal in the active model is far sharper than in the passive model. An interesting feature of the active case is the presence of oscillations off the resonance peaks. To understand these oscillations we observe that in the active case the {\em effective} friction, at least away from the resonance point, namely for $\dot{v}$ small, is not $r$ but rather $r-\rho$ which is very small. Therefore the BM acts as an underdamped oscillator. It is interesting to note here that after oscillations in the perilymph fluid are transformed as described above into mechanical vibrations of the BM, these vibrations are transformed via a complex set of interactions in the Organ of Corti into oscillations in the endolymph, which then set the inner hair cells into motion, thus activating the auditory nerve. It was recently conjectured \cite{berger20} that this further transformation from solid back to viscous fluid vibrations involves a {\em second} filter. We wonder whether one benefit of this second filter is to suppress the spurious oscillations of the BM.
\begin{figure}
    \begin{center}
    \adjincludegraphics[width=6.5cm, trim={0 4cm 0 4cm}, clip]{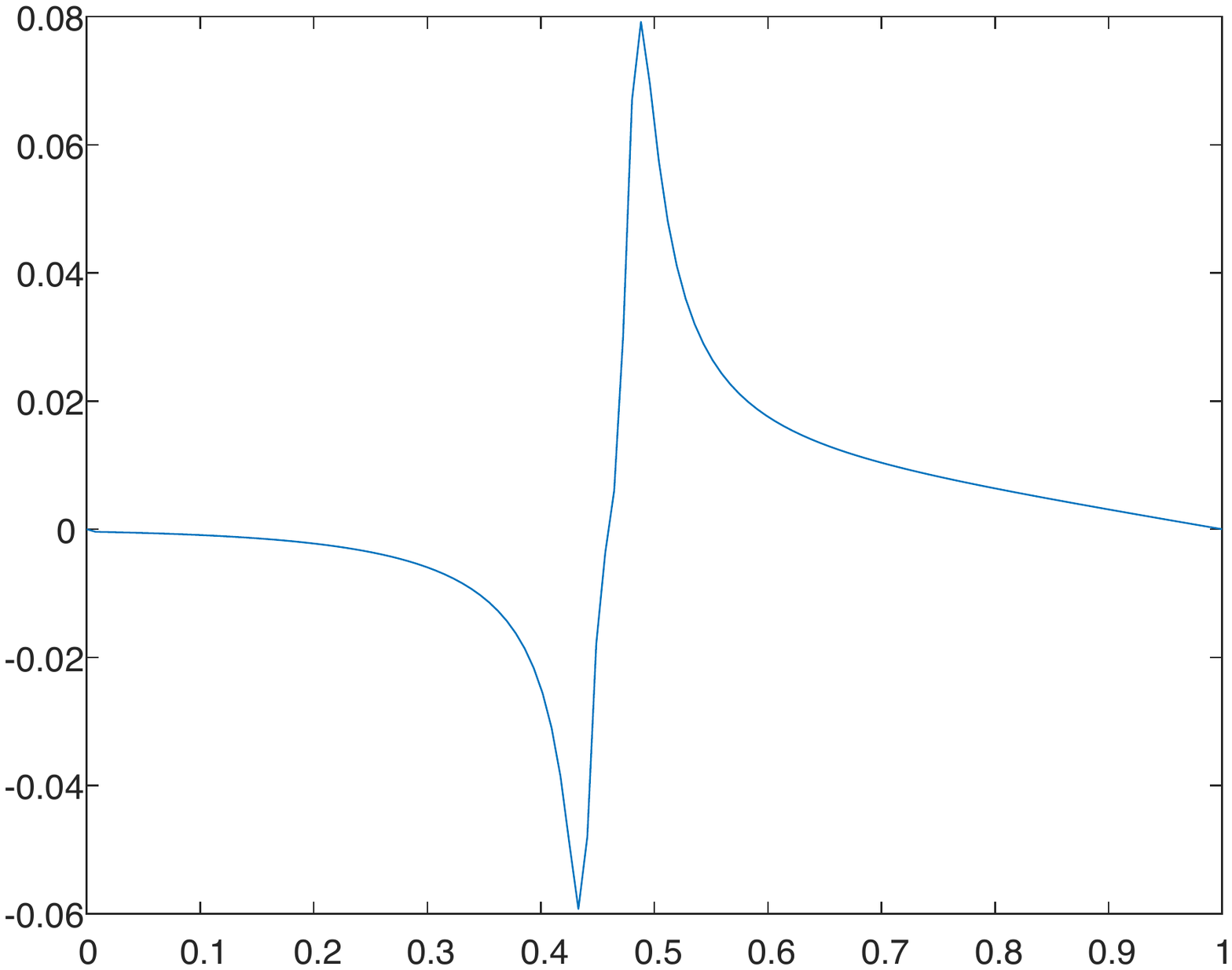}
    \quad
    \adjincludegraphics[width=6.5cm, trim={0 4cm 0 4cm}, clip]{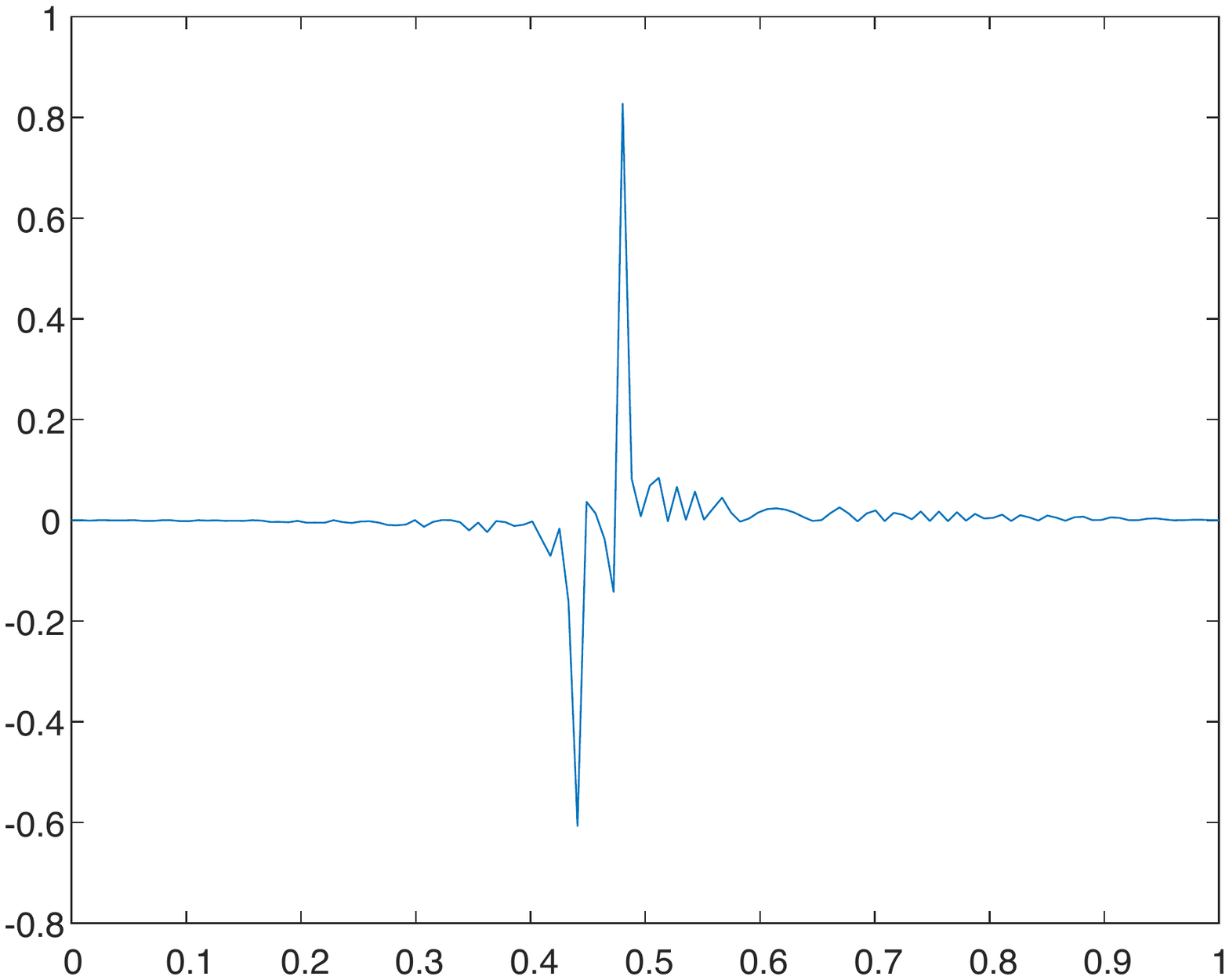}
    \caption{The amplitude $v(x,t)$ for a fixed time $t=T$. The input signal consists of two tones separated by $400$Hz. The friction parameters are $r=0.3, \rho=0.295$. (left) The passive model $\N=0$, (right) The active model.}
    \label{fig:example1}
    \end{center}
\end{figure}

When the friction term $r$ is even larger the passive model cannot distinguish between frequencies separated by $300$Hz. This is demonstrated in Figure 2 (left) where the input is $0.1(\cos(2t)+\cos(2.3t))$. The image depicts the response $v(x,T)$ for the passive model with $r=2,\rho=1.995$. On the other hand, the nonlinear active model easily separates the two tones, as evidenced by Figure 2 (right).

\begin{figure}
    \begin{center}
       \adjincludegraphics[width=6.5cm, trim={0 4cm 0 4cm}, clip]{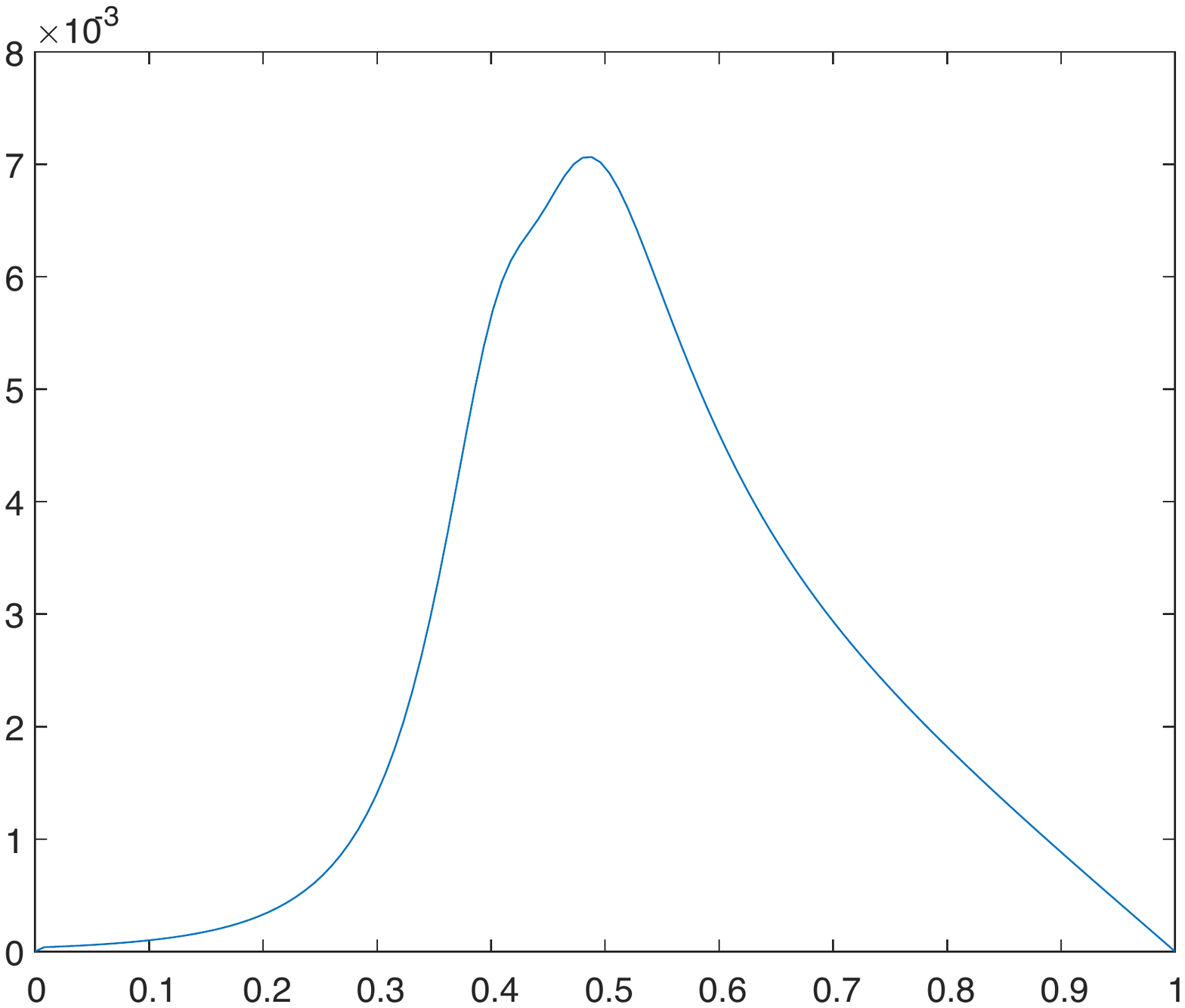}
    \quad
       \adjincludegraphics[width=6.5cm, trim={0 4cm 0 4cm}, clip]{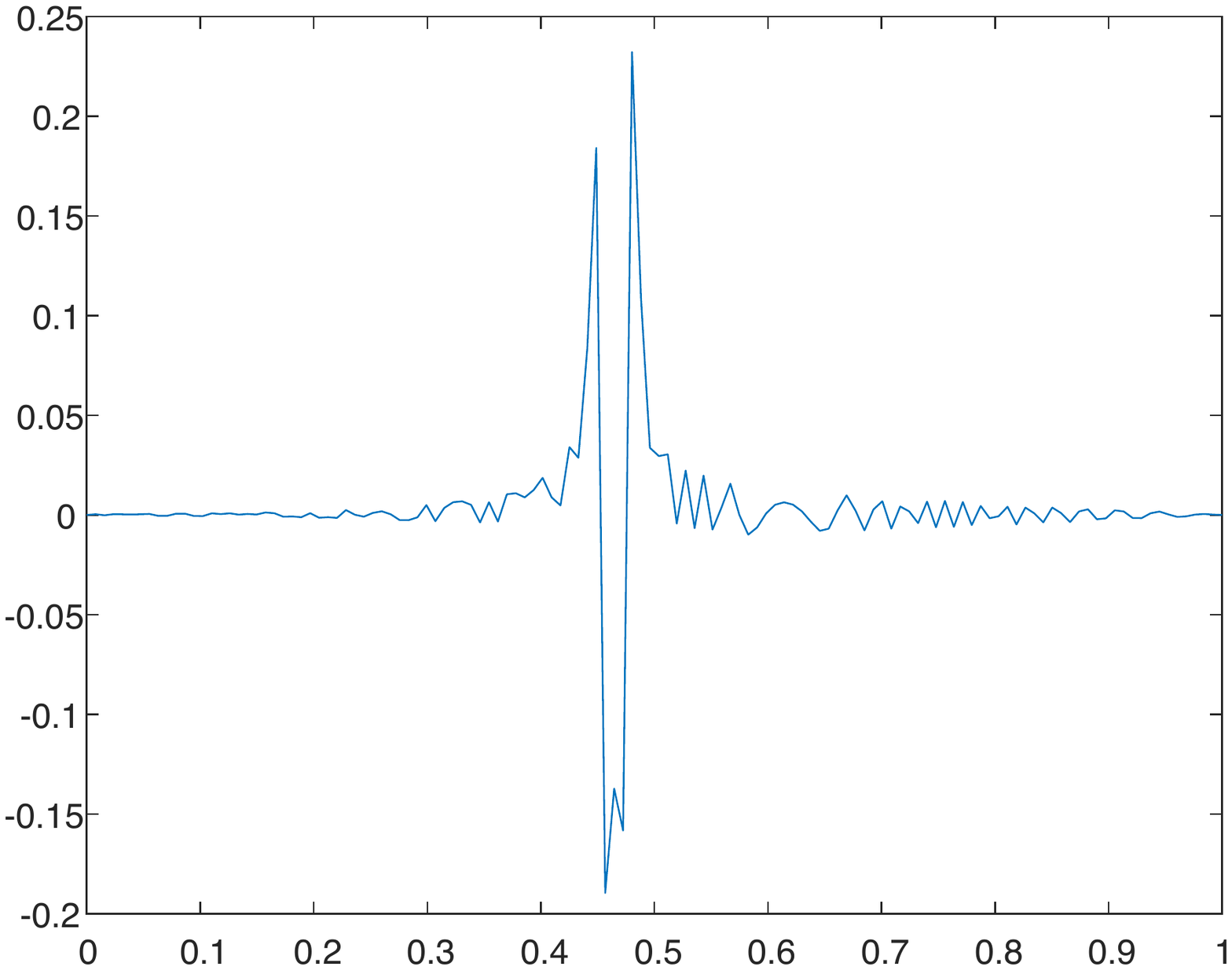}
    \caption{The amplitude $v(x,t)$ for a fixed time $t=T$. The input signal consists of two tones separated by $300$Hz. The friction parameters are $r=2, \rho=1.995$. (left) The passive model $\N=0$, (right) The active model.}
    \label{fig:example2}
    \end{center}
\end{figure}

We point out that the reduced active model is in agreement also with other known features of the cochlea, including otoacoustic emission of sound (this can be observed when $\rho$ is modeled as a random variable with average smaller than $r$ but sufficiently large variance). In addition the active model here provides saturation for high amplitude signals. We comment that our simulations were done at a resolution of $128$ spatial points, which is substantially lower than the resolution provided by the roughly $3000$ inner hair cells in the cochlea.
\newpage
\vskip.2in
\noindent
{\bf Acknowledgments.} J.R. acknowledges the support from a grant supplied by the Israel Science Foundation. P.S. acknowledges the support from a Simons Collaboration grant 585520.

\section{Appendix: The solid-fluid energy estimate}

Recall equations (\ref{Laplace})-(\ref{ic}) and define
\begin{equation}
p(x,z,t) := q(x,z,t)+f(t)(1-x). \label{e1}
\end{equation}
Then, the equation for $q$ is
\begin{equation}
q_{xx}+\dl^{-2} q_{zz}=0, \label{e3}
\end{equation}
with BC
\begin{equation}
q(0,z,t)=q(1,z,t)=q_z(x,1,t)=0,\;\;\; q_z(x,0,t)=\dl^2 \ddot{v}. \label{e5}
\end{equation}

We first derive the energy (and dissipation) balance for the full model. For this purpose we multiply equation (\ref{spring}) by $\dot{v}$ and integrate in $x$ and $t$:
\begin{equation}
\int_0^1 \int_0^T \left( m \ddot{v} \dot{v} + r \dot{v}^2 + kv \dot{v} \right)\; dtdx +
\int_0^1 \int_0^T p(x,0,t)\dot{v} \; dxdt = \int_0^1 \int_0^T \dot{v} N(\dot{v})\; dt dx. \label{e7}
\end{equation}
The terms in the first integral in the left hand side are the BM energy and the dissipation (friction) there, and the last integral is the energy stored in the fluid-BM interaction. Performing one integration we write:
\begin{eqnarray}
\frac{1}{2} \int_0^1 \left(m \dot{v}^2(x,T) + k v^2(x,T) \right) \; dx + \int_0^1 \int_0^T \left(r \dot{v}^2(x,t) + p(x,0,t)\dot{v}\right) \; dxdt =  \nonumber \\
\int_0^1 \int_0^T \dot{v} N(\dot{v})\; dt dx, \label{e9}
\end{eqnarray}
Notice that used here the initial condition (\ref{ic}).

Using the substitution (\ref{e1}) we rewrite the last identity as:
\begin{eqnarray}
\frac{1}{2} \int_0^1 \left(m \dot{v}^2(x,T) + k v^2(x,T) \right) \; dx + \int_0^1 \int_0^T \left(r \dot{v}^2(x,t) + q(x,0,t)\dot{v}\right) \; dxdt =  \nonumber \\
-\int_0^1 \int_0^T f(t)(1-x) \dot{v}(x,t)\;dt dx + \int_0^1 \int_0^T \dot{v} N(\dot{v})\; dt dx. \label{ea1}
\end{eqnarray}

We proceed to estimate the term $\int_0^1\int_0^T q(x,0,t)\dot{v}(x,t) \; dtdx$. Practically, we compute explicitly the NdT operator for the problem (\ref{e3})-(\ref{e5}). In particular we show that this term is positive. We thus expand $q$ of equation (\ref{e1}) into
\begin{equation}
q(x,z,t)=\sum_n a_n(t) \sin (n \pi x) \cosh  (n \pi \dl(1-z)). \label{e13}
\end{equation}
Similarly we expand
\begin{equation}
\ddot{v}(x,t) = \sum_n b_n(t) \sin (n \pi x). \label{e15}
\end{equation}
Equating $q_z(x,0,t) = -\dl^2 \ddot{v}(x,t)$, we obtain
\begin{equation}
a_n(t) = \frac{\dl b_n(t)}{n \sinh(n \pi \dl)}. \label{e17}
\end{equation}
We therefore obtain
\begin{equation}
q(x,0,t) = \sum_n \frac{\dl}{n} \coth (n \pi \dl) b_n(t) \sin(n \pi x). \label{e19}
\end{equation}
Expanding also
\begin{equation}
\dot{v}(x,t) = \sum_n \gm_n(t) \sin(n \pi x), \label{e21}
\end{equation}
and recalling $b_n(t)=\gm_n'(t)$, we obtain
\begin{equation}
\int_0^1 \int_0^T q(x,0,t) \dot{v}(x,t)dx = \frac{\pi^2}{2} \sum_n\frac{\dl \coth (n \pi \dl)}{n} \gm_n^2(T), \label{e23}
\end{equation}
where we used again the initial condition (\ref{ic})). We observe that the sequence $\frac{\dl \coth (n \pi \dl)}{n}$ is positive and bounded. This completes this formal derivation of our energy identity.

Once we have established a positive sign for the solid-fluid interaction energy one can use equation (\ref{ea1}) to obtain an
alternative derivation of Corollary \ref{vbds}, and with further steps that we do not spell out
obtain an alternative convergence proof for the validity of the reduced one-dimensional limit.

\end{document}